\documentclass[a4paper,11pt,reqno]{amsart}
\usepackage{marginnote}
\usepackage{amsmath}
\usepackage{mathtools}
\usepackage[margin=1.45in]{geometry}
\usepackage{amsfonts}
\usepackage{amscd}
\usepackage{amssymb}
\usepackage{graphicx}
\usepackage{mathrsfs}
\usepackage{dsfont}
\usepackage{tikz}

\makeatletter
\g@addto@macro\bfseries{\boldmath}
\makeatother

\usepackage[dvips,all,arc,curve,color,frame]{xy}
\usepackage{enumitem}

\usepackage[colorlinks]{hyperref}
\usepackage[nameinlink]{cleveref}
\usepackage{tikz,mathrsfs}
\usetikzlibrary{arrows,decorations.pathmorphing,decorations.pathreplacing,positioning,shapes.geometric,shapes.misc,decorations.markings,decorations.fractals,calc,patterns}

\newcommand{\f}{\frac}
\newcommand{\tf}{\tfrac}

\newtheorem{thm}{Theorem}[section]

\newtheorem{lemma}[thm]{Lemma}
\newtheorem{claim}[thm]{Claim}
\newtheorem{defin}[thm]{Definition}

\newtheorem{obs}[thm]{Observation}

\theoremstyle{definition}

\begin{document}
\author{Vojt\v{e}ch Dvo\v{r}\'ak}
\address[Vojt\v{e}ch Dvo\v{r}\'ak]{Department of Pure Mathematics and Mathematical Statistics, University of Cambridge, UK}
\email[Vojt\v{e}ch Dvo\v{r}\'ak]{vd273@cam.ac.uk}

\author{Harry Metrebian}
\address[Harry Metrebian]{Department of Pure Mathematics and Mathematical Statistics, University of Cambridge, UK}
\email[Harry Metrebian]{rhkbm2@cam.ac.uk}

\title[A new upper bound for the Ramsey number of fans]{A new upper bound for the Ramsey number of fans
}

%\address[add1]{Department of Pure Mathematics and Mathematical Statistics, Centre for Mathematical Sciences, University of Cambridge, Wilberforce Road, Cambridge CB3 0WB, United Kingdom.}

\begin{abstract} 
A fan $F_n$ is a graph consisting of $n$ triangles, all having precisely one common vertex. Currently, the best known bounds for the Ramsey number $R(F_n)$ are $9n/2-5 \leq R(F_n) \leq 11n/2+6$, obtained by Chen, Yu and Zhao. We improve the upper bound to $31n/6+O(1)$.
\end{abstract}

\maketitle

\section{Introduction}\label{sect1}

Let $G,H$ be graphs. The \textit{Ramsey number} $R(G,H)$ is the smallest positive integer $N$ such that if we colour the edges of the complete graph $K_N$ in two colours, the colouring must contain a copy of $G$ in the first colour or the copy of $H$ in the second colour. When $G$ and $H$ are the same graph, we simply denote this as $R(G)$.

A \textit{fan} $F_n$ is a graph on $2n+1$ vertices with a vertex $v$, called the \textit{centre} of the fan, and $2n$ other vertices $v_1,...,v_{2n}$ such that for $i=1,...,n$, $v v_{2i-1} v_{2i}$ is a triangle. Each of the $n$ edges $v_{2i-1}v_{2i}$ is called a \textit{blade} of the fan. 

Ramsey numbers of fans $R(F_m,F_n)$ have been studied, both in the diagonal case (when $m=n$) and the off-diagonal case. For results in the off-diagonal case, see \cite{li1996fan,lin2009ramsey,lin2010ramsey,zhang2015note}. Lin and Li \cite{lin2009ramsey} also gave the upper bound $6n$ for the Ramsey number $R(F_n)$. A trivial lower bound of $R(F_n) > 4n$ is given by the complete bipartite graph $K_{2n,2n}$. Recently, Chen, Yu and Zhao \cite{chen2021improved} improved the upper bound for $R(F_n)$ significantly and also obtained the first non-trivial lower bound when they proved that $\f{9}{2}n-5 \leq R(F_n) \leq \f{11}{2}n+6$.

As our main result, we make a further improvement to the upper bound, decreasing it from $5.5n$ to about $5.167n$.

\begin{thm}\label{mainresult}
For every $n \geq 1$, we have
$$ R(F_n) \leq \tf{31}{6}n+15. $$
\end{thm}

We are almost certain that $\f{31}{6}n$ is not the true asymptotic magnitude of $R(F_n)$, and hence we make no attempts to optimise the additive constant in the expression above.

%\hm{rewrite this paragraph} 
%Chen, Yu and Zhao \cite{chen2021improved} further conjectured that $R(F_n) \leq 5n=R(nK_3)$, where $nK_3$ is a graph consisting of $n$ disjoint triangles. It is still unknown whether this is indeed the case. This would complement the result of Chen, Yu and Zhao that $R(B_n) \leq R(F_n)$ for sufficiently large $n$, where by $B_n$ we denote a graph consisting of $n$ triangles, all having two vertices in common. We note that both $R(nK_3)$ \cite{burr1975ramsey} and $R(B_n)$ \cite{rousseau1978ramsey} are well understood, so determining magnitude of $R(F_n)$ can be treated in some sense as solving the last missing case.

As well as $F_n$, which consists of $n$ triangles with one vertex in common, one could also consider $nK_3$ and $B_n$, which are graphs consisting of $n$ triangles with precisely zero and two vertices in common, respectively. In contrast to $R(F_n)$, the Ramsey numbers $R(nK_3)$ and $R(B_n)$ are better understood: Burr, Erd{\H{o}}s and Spencer \cite{burr1975ramsey} showed that $R(nK_3) = 5n$ for $n \geq 2$, and Rousseau and Sheehan \cite{rousseau1978ramsey} showed that $R(B_n) \leq 4n+2$ for all $n$ and that this bound is tight for infinitely many $n$. Chen, Yu and Zhao's lower bound $R(F_n) \geq \f{9}{2}n-5$ therefore implies $R(B_n) < R(F_n)$ for sufficiently large $n$. This, together with the observation that $|V(B_n)| < |V(F_n)| < |V(nK_3)|$, led them to speculate that $R(F_n) \leq R(nK_3) = 5n$, but they were unable to show this.

%Consider the Ramsey numbers of $nK_3$, $F_n$ and $B_n$: graphs consisting of $n$ triangles that have precisely zero, one and two vertices in common, respectively. Ramsey numbers of both $nK_3$ \cite{burr1975ramsey} and $B_n$ \cite{rousseau1978ramsey} are well understood. Chen, Yu and Zhao believed that since $$ |V(B_n)|<|V(F_n)|<|V(nK_3)|, $$ it should be true that $$R(B_n) \leq R(F_n) \leq R(nK_3)$$ for sufficiently large $n$. They were able to show that $R(B_n) \leq R(F_n)$ for large $n$, but not that $R(F_n) \leq R(nK_3)=5n$. 

Our approach builds on the ideas of Chen, Yu and Zhao \cite{chen2021improved}: we aim to find large cliques in the graph and then ‘cover' them in a suitable sense. The first new crucial idea in our paper is that of controlling the degrees of vertices in each colour: for that, we use Lemma \ref{cruciallemma}. The proof of this lemma is essentially analogous to the proof of a key lemma of Chen, Yu and Zhao, but using this more general version turns out to be very beneficial. In fact, using the techniques of Chen, Yu and Zhao, this lemma alone can be used to obtain $R(F_n) \leq \f{16}{3}n+O(1)$. 

To go further, we must also introduce a different, more global approach in the latter parts of the proof. We assume we have no $F_n$ of either colour and usually find several large, suitably related cliques and exploit these relations to obtain a contradiction.

The rest of the paper is organised as follows. In Section \ref{sect2}, we introduce our notation and summarise several basic results and lemmas that we will use. In Section \ref{sect3}, we give a brief, non-technical overview of our proof. In Section \ref{sect4}, we go through the technical details of the proof. Finally, in Section \ref{sect5}, we briefly outline further directions of research.

\section{Preliminaries and notation}\label{sect2}

We use standard graph theoretic notation throughout. For a simple graph $G$, we denote its vertex set by $V(G)$ and its edge set by $E(G)$. For $A \subset V(G)$, we write $G[A]$ for the induced subgraph on $A$, and we denote $V(G) \setminus A$ by $\overline{A}$. On the other hand, for a graph $H$, we will write $\overline{H}$ to mean a copy of $H$ consisting of non-edges instead of edges. 

For $v \in V(G)$, we write $N(v)= \lbrace  w \in V(G) \, | \, vw \in E(G) \rbrace$. More generally, for $S \subset V(G)$, we denote $N(S)=\cup_{v \in S} N(v)$, and $N_W(S)=\cup_{v \in S} N_W(v)$. 

Throughout, instead of a two-colouring, we will consider a graph $G$ on $\left\lfloor \f{31}{6}n+15 \right\rfloor$ vertices in the usual graph theoretic sense, and we show that we can always find either $F_n$ or $\overline{F_n}$ inside it. This will be done by contradiction: assume from now on that we cannot find $F_n$ or $\overline{F_n}$ inside $G$. We will examine $G$ more and more thoroughly until we are able to reach the desired contradiction.

As the role of the colours is analogous to us, we will sometimes refer to non-edges as white edges and to edges as black edges. We also sometimes refer to an independent set of vertices as a white clique and to a clique in the usual sense as a black clique. Accordingly, we write $N_W(v) = \lbrace  w \in V(G)\setminus\{v\} \, | \, vw \notin E(G) \rbrace$ for the white neighbourhood of $v$, and similarly $N_W(S)=\cup_{v \in S} N_W(v)$ for $S \subset V(G)$.

% both used also by Chen, Yu and Zhao \cite{chen2021improved}.

For a graph $H$, denote by $\nu(H)$ the size of the largest matching of $H$. Let us recall the following classical result in graph theory due to Hall.
%and by $o(H)$ the number of odd sized components of $H$.

\begin{thm}\label{hall}
Let $H$ be a bipartite graph on parts $X$ and $Y$. For any non-negative integer $d$, $\nu(H) \geq |X|-d$ if and only if $|N(S)| \geq |S|-d$ for every $S \subset X$.
\end{thm}

%\begin{thm}\label{tutte}
%Let $H$ be a graph of order $N$. For any non-negative integer $d$, $\nu(H) \geq \frac{N-d}{2}$ if and only if $o(H-S) \geq |S|+d$ for every subset $S$ of $V(G)$.
%\end{thm}

The value $|N(S)|-|S|$ is known as the \textit{deficiency} of $S$. For a matching $M$ from $X$ to $Y$, we will also refer to $|X|-|M|$ as the deficiency of $M$, where $|M|$ is the number of edges in $M$. Theorem \ref{hall} therefore states that there exists a matching from $X$ to $Y$ of deficiency at most $d$ if and only if every $S \subset X$ has deficiency at most $d$.

Chen, Yu and Zhao \cite[Lemma 1.2]{chen2021improved} show that for any integers $m,n,N$ with $N=4n+m+ \left\lfloor \frac{6n}{m} \right\rfloor +1$, any graph on $N$ vertices contains $F_n$, $\overline{F_n}$, $K_m$ or $\overline{K_m}$. %Their proof uses the fact that a vertex has either at least $\frac{N-1}{2}$ neighbours or at least $\frac{N-1}{2}$ non-neighbours. If we replace this step at the start of the proof by a specific assumption how many neighbours or non-neighbours certain vertex has, the proof can trivially be modified to instead prove the following stronger result, which we will use to handle the case of some vertex having unbalanced degree.
They begin their proof by noting that any vertex has either at least $\frac{N-1}{2}$ neighbours or at least $\frac{N-1}{2}$ non-neighbours. They then show that if a vertex $v$ has degree at least $\frac{N-1}{2}$, then $N(v)$ contains either $nK_2$, $\overline{F_n}$, $K_m$ or $\overline{K_m}$, implying their desired result since the same argument can be applied with colours reversed. By applying their argument to a general graph instead of $N(v)$, we obtain the following result, which we will use throughout our proof.

\begin{lemma}\label{cruciallemma}
Let $H$ be a graph on $3n-c+4$ vertices, where $0<c<\f{5}{8}n$. Then $H$ contains either $nK_2$, $\overline{F_n}$, $K_{2n-2c}$ or $\overline{K_{2n-2c}}$.
\end{lemma}

Next we shall prove a simple lemma.

\begin{lemma}\label{otherimportantlemma}
Let $H$ be a graph on $2k$ vertices. Suppose that $V(H)$ is the disjoint union of $A$ and $B$, each of size $k$, where $H[A]$ is a clique and $H[B]$ is an empty graph. Then $H$ contains $F_{\left\lceil \f{3}{4}k-\f{3}{2} \right\rceil}$ or $\overline{F_{\left\lceil \f{3}{4}k-\f{3}{2} \right\rceil}}$.
\end{lemma}

\begin{proof}
Without loss of generality, we have
$$ d=\max_{v \in A} |N(v) \cap B| \geq \max_{w \in B}|\overline{N(w)} \cap A|. $$

Moreover, clearly $d \geq \f{k}{2}$. Let $z \in A$ be such that $|N(z) \cap B|=d$.

If there is a matching of deficiency at most $d-\f{k}{2}$ from $N(z) \cap B$ to $A \setminus \lbrace z \rbrace$, then clearly $H$ contains $F_{\left\lceil \f{3}{4}k-\f{3}{2} \right\rceil}$ with centre $z$. So assume no such matching exists. Then in particular, by Theorem \ref{hall}, there exists $U \subset N(z) \cap B$ with $$|N(U) \cap A| < |U|-(d-\tf{k}{2}).$$

Now, as $d=\max_{v \in A} |N(v) \cap B| \geq \max_{w \in B}|\overline{N(w)} \cap A|$, and as $U$ is non-empty, we may pick a  vertex $u \in U$, for which we have $$ |N(U) \cap A| \geq |N(u) \cap A| \geq k-d. $$

Hence we get $$ |U| > |N(U) \cap A|+(d-\tf{k}{2}) \geq \tf{k}{2}. $$

From $$ d \geq |U| > |N(U) \cap A|+d-\tf{k}{2}, $$
we also get $|N(U) \cap A| < \f{k}{2}$, and hence $|A \setminus N(U)| > \f{k}{2}$. The bounds on the sizes of $U$ and $A \setminus N(U)$, combined with the observation that there are no edges between these two sets, now give $\overline{F_{\left\lceil \f{3}{4}k-\f{3}{2} \right\rceil}}$ centred at $u$, with at least $\f{k}{2}-\f{1}{2}$ non-central vertices in $A \setminus N(U)$ and the rest in $B \setminus \lbrace u \rbrace$.
\end{proof}

%Now consider any red and blue two-coloring of our complete graph $K_N$ with $N=\frac{31}{6}n+5000$ for $n$ large enough and assume it contains no $F_n$ of either colour. For any set $S \subset V(K_N)$, denote by $N_B(S),N_R(S)$ its blue and red open neighbourhoods (i.e. these don't include $S$) and by $\overline{N_B(S)}$ and $\overline{N_R(S)}$ their complements in $V(K_N) \setminus S$ (i.e. once again not including $S$).

Suppose that $G$ does not contain a copy of $F_n$, and let $A$ be a clique in $G$ such that $|A|>n$ and every vertex of $A$ has degree more than $2n$ in $G$. Let $v$ be a vertex of $A$ with degree $d(v)$. We construct the sets $S(v,A)$ and $C(v,A)$, in a manner analogous to Chen, Yu and Zhao \cite{chen2021improved}, but slightly more general. %Consider $N(v) \setminus A$. Assuming there is no $F_n$ in $G$, we know by Theorem \ref{hall} that there exists $U \subset N(v) \setminus A$ with $$ |U| \geq |N(U) \cap A|+d(v)-2n. $$ 

Let $M$ be a maximal matching in $G[N(v)\setminus A]$, and let $M'$ be a matching of largest size between the independent set $N(v)\setminus (A\cup V(M))$ and $A\setminus \{v\}$. Write $m$ and $m'$ for the number of edges in $M$ and $M'$ respectively. The edges of $M$ and $M'$ form the blades of a fan centred at $v$, and we can pair up all but at most one of the remaining vertices of $A\setminus \{v\}$ into additional blades. We must therefore have $2m+m'+|A|-1 \leq 2n-1$, so $m' \leq 2n-|A|-2m$.

Note that $|N(v)\setminus (A\cup V(M))| = d(v)+1-|A|-2m$. Theorem \ref{hall} and our bound on $m'$ now imply that there exists a set $S(v,A) \subset N(v)\setminus (A\cup V(M))$ with $|S(v,A)| \geq |N(S(v,A)) \cap (A\setminus \{v\})| + d(v)+1-2n$, that is, $$|S(v,A)| \geq |N(S(v,A)) \cap A|+d(v)-2n.$$

Moreover, we can insist that $S(v,A)$ has minimal size among all the sets satisfying the inequality above. Note that since $S(v,A)$ is contained in $N(v)\setminus (A\cup V(M))$, it is an independent set. For convenience, we write $C(v,A)=N(S(v,A)) \cap A$, so we have $$|S(v,A)| \geq |C(v,A)| + d(v) - 2n. $$

We can apply the same argument when $A$ is a white clique. In this case, we consider white edges instead of edges, white degree instead of degree, and so on. We still denote the resulting sets by $S(v,A)$ and $C(v,A)$; it will be clear from the context whether we are working with white or black edges.

%use completely analogous notation for a white clique, with the colours switched, so that instead of the usual degree we consider the white degree, and so on.

Note the following property, which follows directly from the fact that $|S(v,A)| \leq d(v)+1-|A|$, combined with the inequality above relating $|S(v,A)|$ and $|C(v,A)|$:

\begin{obs}\label{boundingc}
We have $|C(v,A)| \leq 2n+1-|A|$.
\end{obs}

We need the notion of \textit{coverability} (again analogous to a concept introduced by Chen, Yu and Zhao \cite{chen2021improved}). 

\begin{defin}\label{coverability}
Let $A$ be a monochromatic clique such that $n<|A|<2n+1$. For $t \geq 1$, we say $A$ is $t$-coverable if $t$ is the smallest integer for which there exists a sequence $v_1,...,v_t$ of vertices of $A$ with the following properties:
\begin{itemize}
    \item $\cup_i C(v_i,A)=A.$
    \item For $i=2,...,t$, we have $v_i \notin \cup_{j<i} C(v_j,A)$.
    \item For $i=1,...,t$, we have $$|C(v_i,A) \setminus \cup_{j<i} C(v_j,A)| \geq |C(z,A) \setminus \cup_{j<i} C(v_j,A)|$$ for any vertex $z$ of $A$ with $z \notin \cup_{j<i} C(v_j,A)$.
\end{itemize}%there exists a sequence $v_1,...,v_t$ of vertices of $A$ such that $\cup_i C(v_i,A)=A$, for $i=2,...,t$ we have $v_i \notin \cup_{j<i} C(v_j,A)$ and moreover $|C(v_i,A) \setminus \cup_{j<i} C(v_j,A)| \geq |C(z,A) \setminus \cup_{j<i} C(v_j,A)|$ for any vertex $z$ of $A$ with $z \notin \cup_{j<i} C(v_j,A)$.
\end{defin}

So, for example, $A$ is $2$-coverable if there exist $v_1, v_2 \in A$ where $|C(v_1,A)|$ is maximal over all $|C(v,A)|$, and $v_2$ is such that $v_2 \notin C(v_1,A)$ and $C(v_1,A) \cup C(v_2,A) = A$.

Note the following simple properties.
\begin{obs}\label{coveringproperties}
We have that:
\begin{itemize}
    \item For any $j_1<j_2$, $$|C(v_{j_1},A) \setminus \cup_{i<j_1} C(v_i,A)| \geq |C(v_{j_2},A) \setminus \cup_{i<j_1} C(v_i,A)|.$$
    \item For any $j_1 \neq j_2$, the sets $S(v_{j_1},A)$ and $S(v_{j_2},A)$ are disjoint.
    \item If $|A| > \f{k-1}{k}(2n+1)$ and $A$ is $t$-coverable, then $t \geq k$.
\end{itemize}
\end{obs}

The final point above follows from Observation \ref{boundingc}, and implies that no clique satisfying the conditions of Definition \ref{coverability} is $1$-coverable.
%In particular, that implies that  Further note that by Observation \ref{boundingc}, no such clique is $1$-coverable.

\section{Overview of the rest of the proof}\label{sect3}

The rest of the proof is quite technical, so we first summarise the general strategy. There are five cases.

Call a monochromatic clique $A$ \textit{big} if $|A| \geq \f{7}{6}n+5$ and call it \textit{significant} if $|A| \geq n+1$.

In subsections \ref{subse1} and \ref{subse2}, we handle the easier cases when either some vertex has very unbalanced degrees (i.e.\ a much larger degree in one colour than the other) or some significant clique of either colour is $t$-coverable for some $t \geq 4$. Lemma \ref{cruciallemma} and the strategy of Chen, Yu and Zhao \cite{chen2021improved} suffice to tackle these cases.

The next three cases, where all the vertices have quite balanced degrees and all significant cliques are $2$-coverable or $3$-coverable, form the heart of the proof.

In subsection \ref{subse3}, there is still a vertex with slightly unbalanced degrees, forcing the existence of a very large (and in particular big) $3$-coverable clique, and in subsection \ref{subse4}, the degrees are balanced but we assume there is some big $3$-coverable clique. The proofs of these cases follow a very similar argument. Both times, we start with the clique $A$ (black without loss of generality) and its $3$-covering $v_1,v_2,v_3$. We then argue that there must be a large black clique $T$ disjoint from $A$ in $N_W(v_3)$ which satifies certain properties: otherwise, we would find $\overline{F_n}$ centred at $v_3$. Then we take any $z \in T$ and argue that $S(v_1,A) \cup S(v_2,A)$ must contain a large white clique $C$, containing at least one element from each of these sets; else we would find $F_n$ centred at $z$. Finally, we conclude that there must be $\overline{F_n}$ centred at some $a \in C$.

In subsection \ref{subse5}, we consider the final case where all vertices have balanced degrees and every big clique is $2$-coverable. We start by using Lemma \ref{cruciallemma} to find two significant cliques $A$ and $B$ of the same colour, without loss of generality black, with $A$ moreover being big. We consider a $2$-covering $v_1,v_2$ of $A$ and a $2$- or $3$-covering $\lbrace w_i \rbrace$ of $B$. We then show there must exist $i$ such that $S(v_1,A)$ and $S(w_i,B)$ intersect. Finally we fix some $a$ in this intersection and find $\overline{F_n}$ centred at it.

\section{Proof of Theorem \ref{mainresult}}\label{sect4}

Now we prove Theorem \ref{mainresult}. 
%The results we have proven in the previous section help us to reduce our general case to several quite bit more specific ones, but with those, some calculations and case work are still needed. 
Let $G$ be a graph with at least $\left\lfloor\f{31}{6}n + 15\right\rfloor$ vertices, and suppose for contradiction that $G$ does not contain a copy of $F_n$ or $\overline{F_n}$. Throughout, denote $$d=\max \left\lbrace \max_v |N(v)|, \max_w \left|N_W(w) \right| \right\rbrace, $$ i.e.\ $d$ is the larger of the maximum degree and non-degree in our graph $G$.

As discussed in Section \ref{sect3}, we consider five separate cases.

\subsection{\em{$d \geq \f{11}{4}n+5$}}\label{subse1}
\hfill\\

If $d>3n$, contradiction follows immediately from the result of Lin and Li~\cite{lin2009ramsey} that $R(nK_2, F_n) = 3n$: consider the neighbourhood of a vertex of degree $d$ in some colour.

If $\f{11}{4}n+5 \leq d \leq 3n$, by applying Lemma \ref{cruciallemma} to the neighbourhood of a vertex of degree $d$ in some colour, we find that our graph contains a black or white clique $A$ of size at least $2d-4n-8$. This in particular is more than $\f{3}{2}n+1$, so by Observation \ref{coveringproperties}, we know this clique is $t$-coverable for some $t \geq 4$.

Now, $v_t$ is the centre of a fan with blades in the sets $S(v_1,A),...,S(v_{t-1},A)$ with the number of vertices at least $$ |S(v_1,A)|+...+|S(v_{t-1},A)|-(t-1). $$

Since for $i=1,...,t-1$, we have
\begin{align*}
|S(v_i,A)| & \geq |C(v_i,A)|+d(v_i)-2n \\
& \geq |C(v_i,A)|+(\tf{31}{6}n+13-d)-2n \\
&\geq |C(v_i,A)|+\tf{19}{6}n-d+13,
\end{align*}
and $\sum_{i=1}^{t-1}|C(v_i,A)| \geq \frac{t-1}{t}|A| \geq \frac{3}{4}|A|$ by Observation \ref{coveringproperties}, this fan has at least $$ \tf{3}{4}|A|+(t-1)(\tf{19}{6}n-d+12) \geq \tf{13}{2}n-\tf{3}{2}d+30 \geq 2n+30$$ vertices. Hence it has more than $n$ blades and contradiction follows.

\subsection{\em{$d<\f{11}{4}n+5$ and some significant clique is $t$-coverable for some $t \geq 4$}}\label{subse2}
\hfill\\

Call this significant (black or white) clique $A$, and recall $|A| \geq n+1$. 

Again, $v_t$ is the root of a fan with blades in sets $S(v_1,A),...,S(v_{t-1},A)$ with number of elements at least $$ |S(v_1,A)|+...+|S(v_{t-1},A)|-(t-1). $$

Since
\begin{align*}
|S(v_i,A)| & \geq |C(v_i,A)|+d(v_i)-2n \\
& \geq |C(v_i,A)|+(\tf{31}{6}n+13-d)-2n \\
&\geq |C(v_i,A)|+\tf{5}{12}n+8
\end{align*}
%$$|S(v_i,A)| \geq |C(v_i,A)|+d(v_i)-2n-1 \geq |C(v_i,A)|+\frac{5}{12}n+480$$ 
for $i=1,...,t-1$, and $\sum_{i=1}^{t-1}|C(v_i,A)| \geq \frac{t-1}{t}|A| \geq \frac{3}{4}|A|$ by Observation \ref{coveringproperties}, this fan has number of elements at least $$ \tf{3}{4}|A|+(t-1)(\tf{5}{12}n+7) \geq 2n+21.$$ So the contradiction follows.

\subsection{\em{$\f{8}{3}n+6 \leq d < \f{11}{4}n+5$ and every significant clique is $2$- or $3$-coverable.}}\label{subse3}
\hfill\\
%Call most extremal degree $8/3n+\delta$, $0<\delta<n/12$. Find WLOG clique $A$, $|A|=\frac{4}{3}n+2\delta$. $A$ is $3$-coverable, $v_1,v_2,v_3$ its $3$-covering.

By applying Lemma \ref{cruciallemma} to the neighbourhood of a vertex of degree $d$ in some colour, there exists a monochromatic clique $A$ such that $|A| \geq \f{4}{3}n+4$, which is black without loss of generality. By Observation \ref{coveringproperties}, $A$ is not $2$-coverable, so as it is significant it must be $3$-coverable. Let $v_1,v_2,v_3$ be its $3$-covering. Note also that Observation \ref{coveringproperties} tells us that $|A| < \f{3}{2}n+1$.

\begin{claim}\label{betterthantrivdegs}
The degrees of $v_1,v_2,v_3$ are all at least $\f{5}{2}n+5$.
\end{claim}

\begin{proof}
Assume not and suppose some $v_i$ has degree less than $\f{5}{2}n+5$. Then it has white degree at least $\f{8}{3}n+8$. Now by Lemma \ref{cruciallemma}, $N_W(v_i)$ contains either a black or white clique $B$ of size at least $\f{4}{3}n+8$. But if this clique is white, by Lemma \ref{otherimportantlemma} applied to $A$ and $B$, $G$ contains $F_n$ or $\overline{F_n}$, a contradiction. Hence the clique is black. But now $A$ and $B$ are disjoint and each share at most one vertex with each of $S(v_1,A)$, $S(v_2,A)$, $S(v_3,A)$, which moreover are mutually disjoint sets too by Observation \ref{coveringproperties}. So $G$ contains at least 
\begin{align*}
|A|+|B|+\sum_i|S(v_i,A)|-6 & \geq |A|+|B|+|A|+3(\tf{31}{6}n+13-d-2n)-6 \\
& \geq 4n+16+3(\tf{5}{12}n+7)-6 \\
& =\tf{21}{4}n+31>\tf{31}{6}n+15
\end{align*}
vertices, which is a contradiction.
%$v_1,v_2,v_3$ all have white degree at most $8/3n$, and hence black at least $2.5n$. For if not, take white neighbourhood of such $v_i$. It contains some $4/3n$ clique. 
%Claim white, then done by Lemma \label{otherclique}. 
%If it was $B$ and black as $A$, have contradiction by
%$$ |A|+|B|+|S_1|+|S_2|+|S_3| \geq 5.25n>5.17n .$$
\end{proof}

So we in fact have $|S(v_i,A)| \geq |C(v_i,A)|+\f{1}{2}n+5$. 

Next we show two simple results that will be useful later.

\begin{claim}\label{twothirds}
We have $\frac{1}{3}|A| \leq |C(v_1,A)| \leq \frac{2}{3}n$ and $|C(v_2,A)|\geq \frac{1}{2}|A \setminus C(v_1,A)| \geq  \frac{1}{3}n$.
\end{claim}

\begin{proof}
This follows immediately from Observations \ref{boundingc} and \ref{coveringproperties}.
\end{proof}
%Also note that $|C(v_1,A)| \leq \frac{2}{3}n$ by Observation \ref{boundingc}, from which it immediately follows that $|C(v_2,A)|\geq \frac{1}{2}|A \setminus C(v_1,A)| \geq  \frac{1}{3}n$. 

%We further also have $|C(v_3,A)| \geq \f{1}{3}n$, for if not, we would have

%$$ |S(v_1,A)|+|S(v_2,A)| \geq |A \setminus C(v_3,A)| +2(\tf{1}{2}n+5) > 2n+2 $$

%and hence we have $\overline{F_n}$ centred at $v_3$ with blades inside $S(v_1,A)$, $S(v_2,A)$. \hm{Did we use this?}

\begin{claim}\label{s1s2}
We have $|S(v_1,A)|+|S(v_2,A)| > \tf{17}{9}n+10.$
\end{claim}

\begin{proof}
Using Observation \ref{coveringproperties} and Claim \ref{betterthantrivdegs}, we have 
\begin{align*}
|S(v_1,A)|+|S(v_2,A)| & \geq |C(v_1,A)|+|C(v_2,A)|+2(\tf{1}{2}n+5) \\ & \geq \tf{2}{3}|A|+n+10 \\ & > \tf{17}{9}n+10\end{align*}
as required.\end{proof}

Now we get to the heart of the proof.

\begin{claim}\label{Texists}
There exists a black clique $T$ in $N_W(v_3)  \setminus (S(v_1,A) \cup S(v_2,A))$ such that with $N_T=N_W(T) \cap (S(v_1,A) \cup S(v_2,A))$, we have $|T|>|N_T|+\f{5}{12}n+6$.
\end{claim}

\begin{proof}
We form a white fan centred at $v_3$, and we show that it has at least $n$ blades. Set $T'=N_W(v_3) \setminus (S(v_1,A) \cup S(v_2,A))$. Let $M$ be a maximal white matching within $T'$, and add blades consisting of the edges of $M$. Next, take a maximal white matching $M'$ from $T'\setminus V(M)$ to $S(v_1,A) \cup S(v_2,A)$ and add blades consisting of $M'$. Finally, add all but at most one of the remaining vertices of $S(v_i,A)$ for $i=1,2$ by pairing them up together within each set.

Note that we have $|N_W(v_3) | \geq \f{29}{12}n+8 $. The blades of our fan contain all of the vertices of $N_W(v_3)$ except for $T'\setminus (V(M)\cup V(M'))$ and at most two vertices of $S(v_1,A) \cup S(v_2,A)$. Hence if $M'$ does not have deficiency at least $\f{5}{12}n+6$, then there are at least $n$ blades, which is a contradiction. The result then follows by Theorem \ref{hall}.
\end{proof}
%Also have $|S_1|+|S_2|>>n$ 
%and hence have black clique $T$ in $\overline{N(v_3)}$ such that if we denote $N_T=N_W(T) \cap (S_1 \cup S_2)$, we have $|T|>|N_T|+\frac{5}{12}n$.

Now denote by $C$ the largest white clique that can be obtained as follows. Start with $S(v_1,A) \cup S(v_2,A)$. Then remove a set $U$ consisting of $|N_T|$ arbitrary vertices. Finally, remove a maximal black matching between $S(v_1,A)\setminus U$ and $S(v_2,A)\setminus U$.

\begin{claim}\label{Cislarge}
We have $|C| \geq |S(v_1,A)|+|S(v_2,A)|-|N_T|-2n+2|T|-6$.
\end{claim}

\begin{proof}
Assume that instead $|C|<|S(v_1,A)|+|S(v_2,A)|-|N_T|-2n+2|T|-6$. We consider two cases according to the size of $T$. First consider the case $|T| \leq n+3$. Pick any $z \in T$, and form a black $F_n$ centred at $z$ as follows. Begin by adding a maximal matching $M$ between $S(v_1,A) \setminus N_T$ and $S(v_2,A) \setminus N_T$. This matching contains at least $n+3-|T|$ edges, else we get a white clique $C'$ larger than $C$ which also satisfies our assumptions. After this, add a maximal matching $M'$ between $T$ and $(S(v_1,A) \cup S(v_2,A)) \setminus (N_T \cup V(M))$. If $M'$ consists of fewer than $|T|-3$ edges, then by Claims \ref{s1s2} and \ref{Texists} we have $$\tf{19}{12}n-4 \geq  |N_T|+|T|+2n-2|T|+2 \geq |S(v_1,A)|+|S(v_2,A)| > \tf{17}{9}n+10,$$ which is absurd. So we have found a fan $F_n$, which is a contradiction.

Now consider the easier case $|T|>n+3$. Pick a vertex $z \in T$, and form a black $F_n$ centred at $z$ with blades consisting of a maximal matching between $T$ and $(S(v_1,A) \cup S(v_2,A)) \setminus N_T$. If this matching contains fewer than $n$ edges, then by Claim \ref{s1s2} and the fact that $|T| \leq d-|S(v_1,A)|-|S(v_2,A)| \leq \f{31}{36}n$, we have $$\tf{47}{36}n \geq 2|T|-\tf{5}{12}n-6 \geq  |N_T|+|T| \geq |S(v_1,A)|+|S(v_2,A)| > \tf{17}{9}n+10,$$ which is again absurd.
\end{proof}

%If $|C|<|S_1|+|S_2|-|N_T|-2n+2|T|$, form black $F_n$ as follows. Pick as center any $z \in T$. 

%So assume $|C|>|S_1|+|S_2|-|N_T|-2n+2|T| \geq \frac{13}{18}n$.

We have $|S(v_1,A)|, |S(v_2,A)| \geq \f{5}{6}n+5>\f{7}{12}n+5 > |N_T|+n+3-|T|$ using Claims \ref{twothirds} and \ref{Texists}. By Claim \ref{Cislarge}, when obtaining $C$ we have erased at most $|N_T|+n+3-|T|$ elements from either set, so this white clique $C$ contains a vertex $a_1 \in S(v_1,A)$ and a vertex $a_2 \in S(v_2,A)$.

Note that by Claims \ref{s1s2}, \ref{Texists} and \ref{Cislarge}, we have 
\begin{align*}
|C| & \geq |S(v_1,A)|+|S(v_2,A)|-|N_T|-2n+2|T|-6\\
& > \tf{17}{9}n+10+2(|T|-|N_T|)-2n-6\\
& > \tf{17}{9}n+10+\tf{5}{6}n+12-2n-6\\
& = \tf{13}{18}n+16.
\end{align*}
%$$|C| \geq \tf{17}{9}n+10+\tf{5}{6}n+12-2n-6 = \tf{13}{18}n+16.$$ 
Consequently, either $|C \cap S(v_2,A)| > \f{1}{6}n$ or $|C \cap S(v_1,A)| > \f{5}{9}n+16.$ We treat these cases separately.

First, assume that $|C \cap S(v_2,A)| > \f{1}{6}n$. We will construct a white fan centred at $a_1$ and show that it has at least $n$ blades. We begin by claiming that $$|S(v_1,A)|>|A \setminus C(v_1,A)|=|A|-|C(v_1,A)|. $$ Indeed this holds, since by Claim \ref{twothirds} we have
\begin{align*}
|S(v_1,A)|-|A|+|C(v_1,A)| &= (|S(v_1,A)|-|C(v_1,A)|)+2|C(v_1,A)|-|A|\\
& \geq \tf{1}{2}n+5+\tf{2}{3}|A|-|A|\\
& > \tf{1}{2}n+5-\tf{1}{3}(\tf{3}{2}n+1)>0.
\end{align*}
%$$|S(v_1,A)|-|A|+|C(v_1,A)|\geq \tf{1}{2}n+5+\tf{2}{3}|A|-|A|>0.$$

So up to at most two vertices, we can use all the vertices of $S(v_1,A), A \setminus C(v_1,A)$ and $C \cap S(v_2,A)$ in our fan, by first taking blades with a vertex in $S(v_1,A)$ and the other in $A \setminus C(v_1,A)$, and then pairing up all but at most one of the remaining vertices in $S(v_1,A)$ and all but at most one of the vertices in $C \cap S(v_2,A)$. But $$|S(v_1,A)|+|A \setminus C(v_1,A)|+|C \cap S(v_2,A)| > \tf{4}{3}n+4+\tf{1}{2}n+5+\tf{1}{6}n = 2n+9,$$ 
so our fan has at least $2n+7$ vertices and therefore at least $n$ blades, which is a contradiction.

Next consider the case $|C \cap S(v_2,A)| \leq \f{1}{6}n$, $|C \cap S(v_1,A) | > \f{5}{9}n+16$. If $|S(v_2,A)| \geq |A \setminus C(v_2,A)|$, we can finish the argument as above, now with $a_2$ as the centre.

So assume $|S(v_2,A)|<|A \setminus C(v_2,A)|$. We construct a black fan centred at $a_2$ and show that it has at least $n$ blades. First add $|S(v_2,A)|-1$ blades with one vertex in $S(v_2,A)$ and one in $A \setminus C(v_2,A)$, and then pair up all but at most one vertex of $C \cap S(v_1,A)$. Using Claim \ref{twothirds}, we find that our fan has at least
$$|S(v_2,A)|-1+\tf{5}{18}n+7 \geq \tf{5}{6}n+\tf{5}{18}n+6>n$$ blades, which is a contradiction. Thus we have shown that if $\f{8}{3}n+6 \leq d < \f{11}{4}n+5$ and every significant clique is $2$- or $3$-coverable, then $G$ contains a monochromatic $F_n$.

\subsection{\em{$d < \f{8}{3}n+6$ and there is a $3$-coverable big clique}}\label{subse4}
\hfill\\
%This will be very very similar to previous subsection.

%Call most extremal degree $31/12n+\delta$, $0<\delta<n/12$. Find WLOG clique $A$, $|A| \geq 7/6n$, that is $3$-coverable, $v_1,v_2,v_3$ its $3$-covering.

By assumption, there exists a monochromatic (without loss of generality black) clique $A$ such that $\f{3}{2}n+1 > |A| \geq \f{7}{6}n+5$ and $A$ is $3$-coverable, with $3$-covering $v_1,v_2,v_3$. As before, the upper bound on $|A|$ comes from Observation \ref{coveringproperties}.

Note that $v_1,v_2,v_3$ all have black degree at least $\f{5}{2}n+7$. So we have $|S(v_i,A)| \geq |C(v_i,A)|+\f{1}{2}n+7$. 

As in the previous subsection, we begin by proving some simple results.

\begin{claim}\label{fivesix}
The following inequalities hold:
\begin{itemize}
    \item $\frac{1}{3}|A| \leq |C(v_1,A)| \leq \frac{5}{6}n$.
    \item $|C(v_2,A)|\geq \frac{1}{2} |A \setminus C(v_1,A)| \geq \frac{1}{6}n$. 
    \item$|C(v_3,A) \setminus(C(v_1,A) \cup C(v_2,A))| \geq \f{1}{6}n$.
\end{itemize}
\end{claim}

\begin{proof}
The first two results follow immediately by Observations \ref{boundingc} and \ref{coveringproperties}. Finally, if we did not have $|C(v_3,A) \setminus(C(v_1,A) \cup C(v_2,A))| \geq \f{1}{6}n$, we would have 
$$ |S(v_1,A)|+|S(v_2,A)| \geq (n+4)+2(\tf{1}{2}n+7) > 2n+2 $$
and hence we would have $\overline{F_n}$ centred at $v_3$ with blades inside $S(v_1,A)$, $S(v_2,A)$. But that is a contradiction.\end{proof}
%Also note that $|C(v_1,A)| \leq \frac{5}{6}n$ by Observation \ref{boundingc}, from which it immediately follows that $|C(v_2,A)|\geq \frac{1}{2} |A \setminus C(v_1,A)| \geq \frac{1}{6}n$. 
%We further also have $|C(v_3,A) \setminus(C(v_1,A) \cup C(v_2,A))| \geq \f{1}{6}n$, for if not, we would have

%$$ |S(v_1,A)|+|S(v_2,A)| \geq (n+4)+2(\tf{1}{2}n+7) > 2n+2 $$

%and hence we have $\overline{F_n}$ centred at $v_3$ with blades inside $S(v_1,A)$, $S(v_2,A)$.

\begin{claim}\label{s1new}
We have $|S(v_1,A)|+|S(v_2,A)| > \tf{16}{9}n+16.$
\end{claim}

\begin{proof}
Using Observation \ref{coveringproperties}, we have
\begin{align*}
|S(v_1,A)|+|S(v_2,A)| & \geq |C(v_1,A)|+|C(v_2,A)|+2(\tf{1}{2}n+7) \\ & \geq \tf{2}{3}|A|+n+14 \\ & > \tf{16}{9}n+16
\end{align*}
as required.
\end{proof}

Now we get to the key parts of the proof.

\begin{claim}\label{Texistsagain}
There exists a black clique $T$ in $N_W(v_3) \setminus (S(v_1,A) \cup S(v_2,A))$ such that with $N_T=N_W(T) \cap (S(v_1,A) \cup S(v_2,A))$, we have $|T|>|N_T|+\f{1}{2}n+5$.
\end{claim}

The proof of Claim \ref{Texistsagain}, which uses Claim \ref{s1new}, is analogous to the proof of Claim \ref{Texists}, and hence is omitted.
%Also $|C_1| \leq \frac{5}{6}n$, $|C_2|,|C_3| \geq \frac{n}{6}$
%Also have $|S_1|+|S_2|>>n$ and hence have black clique $T$ in $\overline{N(v_3)}$ such that if we denote $N_T=N_W(T) \cap (S_1 \cup S_2)$, we have $|T|>|N_T|+\frac{7}{12}n-\delta$.

%Now denote $C$ the largest white clique in $S_1 \cup S_2$ which we can obtain by first removing $N_T$ vertices from $S_1 \cup S_2$, and then removing vertices in pairs, one from each set at time, and both in each pair connected. 

%If $|C|<|S_1|+|S_2|-|N_T|-2n+2|T|$, form black $F_n$ as follows. Pick as center any $z \in T$. Keep adding always a blade consisting of some $s_1 \in S_1 \setminus N_T$, $s_2 \in S_2 \setminus N_T$ with $s_1 \sim s_2$. Can add $|T|-n$ such blades, else get bigger white clique than $C$. After add blades of form $t \in T$, $s \in (S_1 \cup S_2) \setminus N_T$. If couldn't add $T$ such, would have $$\frac{3}{2}n \geq  |N_T|+|T|+2n-2|T| \geq |S_1|+|S_2| \geq \frac{16}{9}n,$$ contradiction by far.

%So assume $|C|>|S_1|+|S_2|-|N_T|-2n+2|T| \geq \frac{10}{9}n-4\delta \geq 7/9n$. 
Now denote by $C$ the largest white clique that can be obtained as follows. Start with $S(v_1,A) \cup S(v_2,A)$. Then remove a set $U$ consisting of $|N_T|$ vertices. Finally, remove a maximal matching between $S(v_1,A)\setminus U$ and $S(v_2,A)\setminus U$. 

\begin{claim}\label{Cislargeagain}
We have $|C| \geq |S(v_1,A)|+|S(v_2,A)|-|N_T|-2n+2|T|-6$.
\end{claim}

The proof of Claim \ref{Cislargeagain} is analogous to the proof of Claim \ref{Cislarge}, and hence is once again omitted.

By Claims \ref{fivesix} and \ref{Texistsagain}, we have $|S(v_1,A)|, |S(v_2,A)| \geq \f{2}{3}n+7>\f{1}{2}n > |N_T|+n+3-|T|$. Claim \ref{Cislargeagain} tells us that we have erased at most $|N_T|+n+3-|T|$ elements from either set when obtaining $C$, and hence we know that $C$ contains a vertex $a_1 \in S(v_1,A)$ and a vertex $a_2 \in S(v_2,A)$.

Note that by Claims \ref{s1new}, \ref{Texistsagain} and \ref{Cislargeagain}, we have
\begin{align*}
|C| &\geq |S(v_1,A)|+|S(v_2,A)|-|N_T|-2n+2|T|-6\\
&> \tf{16}{9}n+16+2(|T|-|N_T|)-2n-6\\
& > \tf{16}{9}n+16+n+10-2n-6\\
& = \tf{7}{9}n+20.
\end{align*}
%$$|C| \geq \tf{16}{9}n+16+n+10-2n-6 = \tf{7}{9}n+20.$$ 
So we either must have $|C \cap S(v_2,A)| > \f{1}{3}n$ or $|C \cap S(v_1,A)| > \f{4}{9}n+20.$ We treat the two cases separately.

%As $|S_1|, |S_2| \geq 2/3n>n/2 \geq |N_T|+n-|T|$, $C$ contains both some $a_1 \in S_1$ and some $a_2 \in S_2$.

First, assume $|C \cap S(v_2,A)| > \f{1}{3}n$. We will construct a white fan centred at $a_1$ and show that it has at least $n$ blades. We claim that $$|S(v_1,A)|>|A \setminus C(v_1,A)|=|A|-|C(v_1,A)|. $$
Indeed this holds, since by Claim \ref{fivesix} we have
\begin{align*}
|S(v_1,A)|-|A|+|C(v_1,A)| &= (|S(v_1,A)|-|C(v_1,A)|)+2|C(v_1,A)|-|A|\\
& \geq \tf{1}{2}n+7+\tf{2}{3}|A|-|A|\\
& > \tf{1}{2}n+7-\tf{1}{3}(\tf{3}{2}n+1)>0.
\end{align*}
%$$|S(v_1,A)|-|A|+|C(v_1,A)|\geq \tf{1}{2}n+7+\tf{2}{3}|A|-|A|>0.$$

So up to at most two vertices, we can use all the vertices of $S(v_1,A), A \setminus C(v_1,A)$ and $C \cap S(v_2,A)$ in our fan, by first taking blades with one vertex in $S(v_1,A)$ and the other in $A \setminus C(v_1,A)$, and then pairing up all but at most one of the remaining vertices in $S(v_1,A)$ and all but at most one of the vertices in $C \cap S(v_2,A)$. But 
$$|S(v_1,A)|+|A \setminus C(v_1,A)|+|C \cap S(v_2,A)| > \tf{7}{6}n+5+\tf{1}{2}n+7+\tf{1}{3}n = 2n+12,$$
and so our fan has at least $2n+10$ vertices and therefore at least $n$ blades, which is a contradiction.

%So we can use up $S_1, A \setminus C_1$ and $C \cap S_2$ all in our fan. But $|S_1|+|A \setminus C_1| \geq 5/3n$ and done.

Next consider the case $|C \cap S(v_1,A) | > \f{4}{9}n+20$ and $|C \cap S(v_2,A)| \leq \f{1}{3}n$. Here we consider two subcases.

%Finally consider two cases. 
If $|C(v_2,A)| \geq \f{5}{18}n$, we construct a white fan centred at $a_2$ and show that it has at least $n$ blades. First form as many blades as possible with one vertex in $S(v_2,A)$ and the other in $A \setminus C(v_2,A)$, and then pair up all but at most one of the vertices in $C \cap S(v_1,A)$. If $|S(v_2,A)|>|A \setminus C(v_2,A)|$, then we can also pair up all but at most one of the remaining vertices in $S(v_2,A)$, and we get a contradiction as in the previous case, but with $v_1$ and $v_2$ interchanged. If instead $|S(v_2,A)| \leq |A \setminus C(v_2,A)|$, then our white fan clearly has at least $$2(|S(v_2,A)|-1)+\tf{4}{9}n+19 \geq 2|C(v_2,A)|+2(\tf{1}{2}n+7)+\tf{4}{9}n+17 >2n+2$$ vertices, so once again it has at least $n$ blades and we reach a contradiction.

%can definitely then for $\overline{F_n}$ get $2|S_2|+4/9n>2n$.

If $|C(v_2,A)|<\f{5}{18}n$, it follows by Observation \ref{coveringproperties} that $$|C(v_3,A) \setminus (C(v_1,A) \cup C(v_2,A))|<\tf{5}{18}n.$$ Hence $|C(v_1,A) \cup C(v_2,A)| \geq \f{8}{9}n$, and then Claims \ref{Texistsagain} and \ref{Cislargeagain} give $$|C| \geq \tf{8}{9}n+2(\tf{1}{2}n+7)+2(\tf{1}{2}n+5)-2n-6 = \tf{8}{9}n+18.$$ So as $|C \cap S(v_2,A)| \leq \f{1}{3}n$, we have $|C \cap S(v_1,A)| \geq \f{5}{9}n+18$. 

Now we form a fan in the complement and show that it has at least $n$ blades. Once again, we pick $a_2$ as the centre. Recall that by Claim \ref{fivesix}, $|C(v_3,A) \setminus(C(v_1,A) \cup C(v_2,A))| \geq \f{1}{6}n$. First form $\f{1}{6}n$ blades by pairing up vertices of $C \cap S(v_1,A)$ and $C(v_3,A) \setminus(C(v_1,A) \cup C(v_2,A))$. Next, pair all but at most one of the remaining vertices of $C \cap S(v_1,A)$ with each other into blades. Next, form as many blades as possible with one vertex in $S(v_2,A)$ and the other in $A \setminus C(v_2,A)$. If some vertices of $S(v_2,A)$ remain, pair all of those except at most one with each other into blades. This means our fan either contains all but at most two vertices of $S(v_2,A),C \cap S(v_1,A),A \setminus C(v_2,A) $, in which case we reach a contradiction as before, or it contains at least 
\begin{align*}2(|S(v_2,A)|-1)+\tf{5}{9}n+17+\tf{1}{6}n & \geq 2|C(v_2,A)|+2(\tf{1}{2}n+7)+\tf{13}{18}n+15\\ & > 2n+2\end{align*} vertices, once again giving the desired contradiction. Therefore if $d < \f{8}{3}n+6$ and $G$ contains a $3$-coverable big clique, then $G$ contains a monochromatic $F_n$.

%Now use up at least $n/6$ vertices of $C_3$ paired with $S_1 \cap C$ vertices, all other $S_1 \cap C$ vertices paired with each other and things in $S_2$ paired to $A$, works out (will fill in details later but easy).

\subsection{\em{$d < \f{8}{3}n+6$ and every big clique is $2$-coverable}}\label{subse5}
\hfill\\
%In particular, note that we now assume every big clique is $2$-coverable.

We start by proving a simple but important claim.

\begin{claim}\label{twocliqq}
There exist two disjoint cliques $A,B$ of the same colour such that $\f{4}{3}n+1 > |A| \geq |B| \geq n+1$ and $|A|+|B|=\left\lfloor\f{7}{3}n+18\right\rfloor$.
\end{claim}

\begin{proof}
First, note that by applying Lemma \ref{cruciallemma} to the neighbourhood of a vertex of degree $d$ in some colour, there is a monochromatic clique $C$ of order at least $2d-4n-8$ in $G$. Since $2d \geq |V(G)|-1$, this clique must be big. Hence, by our assumption, it is $2$-coverable. Let $S(u_1,C),S(u_2,C)$ be a $2$-covering of this clique. We have 
% $$ |S(v_1,C)|+|S(v_2,C)| \geq |C|+2(d-2n) \geq 4d-8n-8 \geq \tf{7}{3}n+84.  $$
\begin{align*}
|S(u_1,C)|+|S(u_2,C)| & \geq |C|+d(u_1)+d(u_2)-4n\\
& \geq 2d-8n-8+2(\tf{31}{6}n+13-d)\\
& = \tf{7}{3}n+18.
\end{align*}

We may assume that $|S(u_1,C)| \geq |S(u_2,C)|$. We can now clearly pick $A \subset S(u_1,C)$ and $B \subset S(u_2,C)$ with $|A| \geq |B|$ and $|A|+|B|=\left\lfloor\f{7}{3}n+18\right\rfloor$. Note that $A$ is a big clique, so it is $2$-coverable. By Observation \ref{coveringproperties}, we must have $|A| < \f{4}{3}n+1$. The result then follows.
\end{proof}

%We start by finding two (WLOG black) cliques $A,B$ with $4/3n \geq |A| \geq |B| \geq n$ and $|A|+|B|=7/3n$. This is done by $2$-covering largest clique in our graph and also by noting any clique larger than $4/3n$ is big but not $2$-coverable (and hence doesn't exist).

%As $B$ need not be big, it may be $3$-coverable. If $B$ was $4+$-coverable, would be previous case.
Without loss of generality, assume $A$ and $B$ are black cliques.

Note that $B$ may be either $2$-coverable or $3$-coverable, as it need not be big, but it is significant and by assumption it is not $t$-coverable for any $t \geq 4$.

Now let $v_1,v_2$ be the covering of $A$, and let $w_1,w_2$ (and possibly also $w_3$) be the covering of $B$.

\begin{claim}\label{notalldisjoint}
There exists $i$ such that $S(w_i,B) \cap S(v_1,A) \neq \emptyset$.
\end{claim}

\begin{proof}
Assume not. Then all the sets $S(w_i,B)$ as well as $S(v_1,A)$ are disjoint independent sets. $A$ and $B$ are disjoint cliques, and hence each of these can share at most one vertex with each of the independent sets. So $G$ has at least 
$$|A|+|B|+|S(v_1,A)|+\sum_i |S(w_i,B)|-8 $$ vertices. But now,
$$ |S(v_1,A)| \geq |C(v_1,A)|+\tf{31}{6}n+13-d-2n \geq \tf{1}{2}|A| +\tf{1}{2}n+7$$ and $$ \sum_i |S(w_i,B)| \geq |B|+2(\tf{1}{2}n+7) .$$ Putting these together, 
\begin{align*}
&|A|+|B|+|S(v_1,A)|+\sum_i |S(w_i,B)|-8\\
& \geq \tf{3}{2}(|A|+|B|)+\tf{1}{2}|B|+\tf{3}{2}n+13 \geq \tf{11}{2}n+38>|V(G)|
\end{align*}
which is a contradiction.
\end{proof}

%If not, we would have $$ |A|+|B|+|S_1|+\sum_j |T_j| \geq 7/3n+13/12n+2n>>31/6n, $$ which would be a contradiction. 

We now fix $S(w_i,B)$ such that $S(w_i,B) \cap S(v_1,A) \neq \emptyset$, and let $a \in S(v_1,A) \cap S(w_i,B)$. We will consider two cases, namely $|B \setminus C(w_i,B)| \geq \f{1}{3}n$ and $|B \setminus C(w_i,B)| < \f{1}{3}n$. In each case, we will construct a white fan centred at $a$ and show that it has at least $n$ blades. We will need a simple claim.

\begin{claim}\label{lastclaim}
The following inequalities hold:
\begin{itemize}
    \item $|S(w_i,B)| \geq \tf{1}{2}n+7$.
    \item $|S(v_1,A)| \geq |A \setminus C(v_1,A)|+\tf{1}{2}n+7$. 
    \item$|S(v_1,A)|+|A \setminus C(v_1,A)| \geq \tf{5}{3}n+15.$
\end{itemize}
\end{claim}

\begin{proof}
For the first inequality, we have
$$|S(w_i,B)| \geq |C(w_i,B)|+\tf{31}{6}n+13-d-2n \geq |C(w_i,B)| + \tf{1}{2}n+7 \geq \tf{1}{2}n+7.$$
We obtain the other two inequalities by applying the same argument to $S(v_1,A)$ to produce $|S(v_1,A)| \geq |C(v_1,A)|+\f{1}{2}n+7$. Observation \ref{coveringproperties} gives that $|C(v_1,A)| \geq \f{1}{2}|A|$, implying the second inequality. For the final inequality, note that $|A| \geq \f{1}{2}\left\lfloor\f{7}{3}n+18\right\rfloor \geq \f{7}{6}n+8$.
\end{proof}

%Note that $$|S(w_i,B)| \geq |C(w_i,B)|+\tf{1}{2}n+7 \geq \tf{1}{2}n+7,$$ that $$|S(v_1,A)| \geq \tf{1}{2}|A|+\tf{1}{2}n+7 \geq |A \setminus C(v_1,A)|+\tf{1}{2}n+7 $$ and that $$|S(v_1,A)|+|A \setminus C(v_1,A)| \geq |A|+\tf{1}{2}n+7 \geq \tf{5}{3}n+15.$$ 

First, assume that $|B \setminus C(w_i,B)| \geq \f{1}{3}n$. We start by picking $\f{1}{3}n$ blades with a vertex in $S(w_i,B)$ and the other in $B \setminus C(w_i,B)$; by Claim \ref{lastclaim}, this can be done.
Next, we keep adding blades with a vertex in $S(v_1,A)$ and the other in $A \setminus C(v_1,A)$ until we run out of vertices in $A \setminus C(v_1,A)$; we know we will run out of vertices in this set, by Claim \ref{lastclaim} and because we used at most $\f{1}{3}n+1$ vertices of $S(v_1,A)$ in the previous step. Finally, we pair all the remaining vertices of $S(v_1,A)$ into blades, except possibly one vertex. 

We have used all the vertices of $A \setminus C(v_1,A)$ in blades and all but at most one of $S(v_1,A)$. Thus, by Claim \ref{lastclaim}, we have used a total of at least $\f{5}{3}n+14$ vertices. But we also used at least $\f{1}{3}n$ vertices of $B \setminus C(w_i,B)$, which we have not yet counted, so clearly our fan has at least $n$ blades and we have reached a contradiction.
%use up all vertices of $A \setminus C_i$ and $S_i$ in blades, can do that too by above. So have at least $2n$ vertices and done.

Next, assume that instead $|B \setminus C(w_i,B)| < \f{1}{3}n$. Note that as $S(v_1,A) \cap S(w_i,B)$ is an independent set, it must have fewer than $\f{4}{3}n+1$ vertices, or else by Observation \ref{coveringproperties} it would not be $2$-coverable, contradicting our assumption about big cliques. Since $S(v_1,A)$ and $S(w_i,B)$ intersect each of $B$ and $A$ in at most one vertex, we have 
\begin{align*}
&|S(v_1,A) \cup S(w_i,B)|+|B \setminus C(w_i,B)|+|A \setminus C(v_1,A)| \\
&\geq |A|+|B|+2(\tf{1}{2}n+7)-|S(v_1,A) \cap S(w_i,B)|-2 \\
&\geq 2n+28.
\end{align*}

Thus if we can show we can find a white fan centred at $a$ using all but at most two of the vertices in $$S(v_1,A) \cup S(w_i,B) \cup (A \setminus C(v_1,A)) \cup (B \setminus C(w_i,B)),$$ then it has at least $n$ blades and we are done.

We start by creating blades with one end in $B \setminus C(w_i,B)$ and the other in $S(w_i,B)$. We eventually run out of vertices in $B \setminus C(w_i,B)$, as $$ |B \setminus C(w_i,B)| < \tf{1}{3}n<\tf{1}{2}n+7 \leq |S(w_i,B)|,$$ 
using Claim \ref{lastclaim} for the last inequality.
%But we can easily use up all of these vertices (and hence win). First use up all in $B \setminus D_i$ by using at most $n/3$ in $T_i$. 
Next, we create blades with one element in $S(v_1,A)$ and the other in $A \setminus C(v_1,A)$. Since we know that $$|S(v_1,A)| \geq |A \setminus C(v_1,A)|+\tf{1}{2}n+7 $$ by Claim \ref{lastclaim}, and we used at most $\f{1}{3}n+1$ vertices of $S(v_1,A)$ in the previous step, we shall first run out of the elements of $A \setminus C(v_1,A)$. Finally, we can use all but at most one of the remaining elements in $S(v_1,A)$ by pairing them up, and we can use all but at most one of the remaining elements in $S(w_i,B) \setminus S(v_1,A)$ by pairing them up. The result follows, finishing the proof of Theorem \ref{mainresult}.

\section{Conclusion}\label{sect5}

In this paper, through controlling the degrees of the vertices as well as taking a more global approach, we have reduced the bound on $R(F_n)$ from $(5+\f{1}{2})n+O(1)$ to $(5+\f{1}{6})n+O(1)$. This is still far from the lower bound of $\f{9}{2}n+O(1)$, which we suspect is much closer to the correct magnitude.

We expect that with some more care the methods in our proof could likely be improved to give an upper bound of $(5+\delta)n+O(1)$ for some $\delta<\f{1}{6}$, but we encounter more obstacles as we approach $5n$. For example, our proof repeatedly makes use of monochromatic cliques of order significantly larger than $n$. However, in a graph of order $5n+O(1)$ in which every vertex has degree around $\f{5}{2}n$ and there is no monochromatic $F_n$, Lemma \ref{cruciallemma} only guarantees the existence of a clique of order $n+O(1)$. It therefore seems unlikely that present methods could bring the upper bound close to $\f{9}{2}n$, or even verify the conjecture of Chen, Yu and Zhao \cite{chen2021improved} that $r(F_n) \leq R(nK_3)=5n$.

\section*{Acknowledgements}

The authors would like to thank their PhD supervisor Professor B\'ela Bollob\'as for his helpful comments.

\bibliographystyle{abbrv}
\bibliography{sample}

\end{document}